\documentclass{amsart}

\usepackage{amsfonts}
\usepackage{amsmath}
\usepackage{hyperref}

\newtheorem{theorem}{Theorem}
\theoremstyle{plain}

\newtheorem{corollary}{Corollary}

\newtheorem{remark}{Remark}

\numberwithin{equation}{section}

\begin{document}

\title[Sharp weighted companion inequalities of Ostrowski type] {NEW SHARP INEQUALITIES OF OSTROWSKI AND
GENERALIZED TRAPEZOID TYPE FOR THE RIEMANN–STIELTJES INTEGRALS AND
APPLICATIONS
}

\author[M.W. Alomari]{Mohammad W. Alomari}

\address{Department of Mathematics,
Faculty of Science, Jerash University, 26150 Jerash, Jordan}
\email{mwomath@gmail.com}
\date{\today}

\keywords{Ostrowski's inequality, bounded variation,
Riemann-Stieltjes integral.
\\
This paper was published in Ukrainian Mathematical
Journal,  65 (7) 2013, 895--916. However, I received an appreciated e-mail from Dr. E. Kikianty, where she found several major erratum and so recommended to publish this version of the paper.}

\begin{abstract}
In this paper, new sharp weighted generalizations of Ostrowski and
generalized trapezoid type inequalities for the Riemann--Stieltjes
integrals are proved. Several related inequalities are deduced and
investigated. New Simpson's type inequalities for
$\mathcal{RS}$--integral are pointed out. Finally, as application;
an error estimation of a general quadrature rule for
$\mathcal{RS}$--integral via Ostrowski--generalized trapezoid
quadrature formula is given.
\end{abstract}


\maketitle

\section{Introduction}
In order to approximate the Riemann--Stieltjes integral $\int_a^b
{f\left( t \right)du\left( t \right)}$, Dragomir \cite{Dragomir2}
has introduced the following (general) quadrature rule:
\begin{align*}
\mathcal{D}\left( f,u;x \right) := f\left( x \right) \left[
{u\left( {b} \right) - u\left( a \right)} \right]  - \int_a^b
{f\left( t \right)du\left( t \right)}
\end{align*}
After that, many authors have studied this quadrature rule under
various assumptions of integrands and integrators. In the
following, we give a summary of these
results: let $f,u:[a,b] \to \mathbb{R}$ be as follow:\\

(1)\,\, $f$ is of $r$-$H_f$--H\"{o}lder type on $[a,b]$, where
$H_f>0$ and $r \in (0,1]$ are given.

(1$^{\prime}$)\,\, $u$ is of $s$-$H_u$--H\"{o}lder type on
$[a,b]$, where $H_u>0$ and $s \in (0,1]$ are given.

(2)\,\, $f$ is of bounded variation on $[a,b]$.

(2$^{\prime}$)\,\, $u$ is of bounded variation on $[a,b]$.

(3)\,\, $f$ is $L_f$--Lipschitz on $[a,b]$.

(3$^{\prime}$)\,\, $u$ is $L_u$--Lipschitz on $[a,b]$.

(4)\,\, $f$ is monotonic nondecreasing on $[a,b]$.

(4$^{\prime}$)\,\, $u$ is monotonic nondecreasing on $[a,b]$.

(5)\,\, $f$ is $L_{1,f}$--Lipschitz on $[a,x]$ and
$L_{2,f}$--Lipschitz on $[x,b]$.

(5$^{\prime}$)\,\, $u$ is $L_{1,u}$--Lipschitz on $[a,x]$ and
$L_{2,u}$--Lipschitz on $[x,b]$.

(6)\,\, $f$ is monotonic nondecreasing on $[a,x]$ and $[x,b]$.

(6$^{\prime}$)\,\,\, $u$ is monotonic nondecreasing on $[a,x]$ and
$[x,b]$.

(7)\,\, $f$ is absolutely continuous on $[a,b]$.

(8)\,\, $|f'|$ is convex on $[a,b]$.

Then, the following inequalities hold under the corresponding
assumptions:
\begin{multline}
\label{eq.Dra1} \left| {\mathcal{D}\left( f,u;x \right)} \right|
\\
\le \left\{ \begin{array}{l}
 H_f\left[ {\frac{{b - a}}{2} + \left| {x - \frac{{a + b}}{2}}
\right|} \right]^r \cdot \bigvee_a^b\left( {u} \right),\,\,\,\,\,\,\,\,\,\,\,\,\,\,\,\,\,\,\,\,\,\,\,\,\,\,\,\,\,\,\,\,\,\,\,\,\,\,\,\,\,\,\,\,\,\,\,\,\,\,\,\,\,\,\,\,\,\,\,\,\,\,\,\,\,\,\,\,\,\,\,\,\,\,\,\,\,\,\,\,\,\,\,\,\,\,\,\,\, (1),(2^{\prime}), \,\,\left( {[13]} \right) \\
  \\
 H_u \left\{ \begin{array}{l} \left[ {\left( {x - a} \right)^s +
\left( {b - x} \right)^s } \right]\left[ {\frac{1}{2}\bigvee_a^b
\left( f \right) + \frac{1}{2}\left| {\bigvee_a^x \left( f \right)
- \bigvee_x^b \left( f
\right)} \right|} \right] \\
\\
\left[ { \left( {x - a} \right)^{qs}  + \left( {b - x}
\right)^{qs} } \right]^{1/q} \left[ {\left( {\bigvee_{a}^x\left( f
\right)} \right)^p  + \left(
{\bigvee_{x}^b\left( f \right)} \right)^p } \right]^{1/p},\,\,\,\,\,\,\,\,\,\,\,\,\, (1^{\prime}),(2),\,\,\left( {[14]} \right)\\
\,\,\,\,\,\,\,\,\,\,\,\,\,\,\,\,\,\,\,\,\,\,\,\,\,\,\,\,\,\,\,\,\,\,\,\,\,\,\,\,\,\,\,\,\,\,\,\,\,\,\,\,\,\,\,\,\,\,\,\,\,\,\,\,\,\,\,\,\,\,\,\,\,\,\,\,\,\,\,\,\,\,\,\,\,\,\,\,\,\,\,\,p>1\,\frac{1}{p}+\frac{1}{q}=1\\
 \left[ {\frac{{b - a}}{2} + \left| {x - \frac{{a + b}}{2}}
\right|} \right]^s  \cdot \bigvee_a^b \left( {f} \right) \\
 \end{array} \right. \\
\\
\frac{{L_uH_f}}{{r + 1}}\left[ {\left( {x - a} \right)^{r + 1}  +
\left( {b - x} \right)^{r + 1} } \right],\,\,\,\,\,\,\,\,\,\,\,\,\,\,\,\,\,\,\,\,\,\,\,\,\,\,\,\,\,\,\,\,\,\,\,\,\,\,\,\,\,\,\,\,\,\,\,\,\,\,\,\,\,\,\,\,\,\,\,\,\,\,\,\,\,\,\,\,\,\,\,\,\,\,\,\,\,\,\,\,\,\,\,\,\,\,\,\,(1),(3^{\prime}), \,\,\left( {[6]} \right)\\
\\
\frac{{L_fH_u}}{{s + 1}}\left[ {\left( {x - a} \right)^{s + 1}  +
\left( {b - x} \right)^{s + 1} } \right],\,\,\,\,\,\,\,\,\,\,\,\,\,\,\,\,\,\,\,\,\,\,\,\,\,\,\,\,\,\,\,\,\,\,\,\,\,\,\,\,\,\,\,\,\,\,\,\,\,\,\,\,\,\,\,\,\,\,\,\,\,\,\,\,\,\,\,\,\,\,\,\,\,\,\,\,\,\,\,\,\,\,\,\,\,\,\,\, (1^{\prime}),(3), \,\,\left( {[6]} \right)\\
\\
L_uL_f\left[ {\frac{1}{4} + \left( {\frac{{x - {\textstyle{{a + b}
\over 2}}}}{{b - a}}} \right)^2 } \right]\left( {b - a}
\right)^2,\,\,\,\,\,\,\,\,\,\,\,\,\,\,\,\,\,\,\,\,\,\,\,\,\,\,\,\,\,\,\,\,\,\,\,\,\,\,\,\,\,\,\,\,\,\,\,\,\,\,\,\,\,\,\,\,\,\,\,\,\,\,\,\,\,\,\,\,\,\,\,\,\,\,\,\,\,\,\,\,\,\,\,\,\,\,(3),(3^{\prime}), \,\,\left( {[6]} \right)\\
\\
\max \left\{ {L_{1,u} ,L_{2,u} } \right\} \times \left\{
\begin{array}{l}
 \left[ {\frac{{b - a}}{2} + \left| {x - \frac{{a + b}}{2}} \right|} \right]\left[ {f\left( b \right) - f\left( a \right)} \right], \\
\,\,\,\,\,\,\,\,\,\,\,\,\,\,\,\,\,\,\,\,\,\,\,\,\,\,\,\,\,\,\,\,\,\,\,\,\,\,\,\,\,\,\,\,\,\,\,\,\,\,\,\,\,\,\,\,\,\,\,\,\,\,\,\,\,\,\,\,\,\,\,\,\,\,\,\,\,\,\,\,\,\,\,\,\,\,\,\,\,\,\,\,\,\,\,\,\,\,\,\,\,\,\,\,\,\,\,\,\,\,\,\,\,\,\,\,(5^{\prime}),(6), \,\,\left( {[6]} \right)\\
 \left[ {\frac{{f\left( b \right) - f\left( a \right)}}{2} + \frac{1}{2}\left| {f\left( x \right) - \frac{{f\left( a \right) + f\left( b \right)}}{2}} \right|} \right]\left( {b - a} \right) \\
 \end{array} \right.\\
 \\
\max \left\{ {L_{1,f} ,L_{2,f} } \right\} \times \left\{
\begin{array}{l}
 \left[ {\frac{{b - a}}{2} + \left| {x - \frac{{a + b}}{2}} \right|} \right]\left[ {u\left( b \right) - u\left( a \right)} \right], \\
 \,\,\,\,\,\,\,\,\,\,\,\,\,\,\,\,\,\,\,\,\,\,\,\,\,\,\,\,\,\,\,\,\,\,\,\,\,\,\,\,\,\,\,\,\,\,\,\,\,\,\,\,\,\,\,\,\,\,\,\,\,\,\,\,\,\,\,\,\,\,\,\,\,\,\,\,\,\,\,\,\,\,\,\,\,\,\,\,\,\,\,\,\,\,\,\,\,\,\,\,\,\,\,\,\,\,\,\,\,\,\,\,\,\,\,\,\,(5),(6^{\prime}), \,\,\left( {[6]} \right)\\
 \left[ {\frac{{u\left( b \right) - u\left( a \right)}}{2} + \frac{1}{2}\left| {u\left( x \right) - \frac{{u\left( a \right) + u\left( b \right)}}{2}} \right|} \right]\left( {b - a} \right) \\
 \end{array} \right.\\
 \\
 H_f\left[ {\frac{{b - a}}{2} + \left| {x - \frac{{a + b}}{2}}
\right|} \right]^r \cdot \left[ {u\left( b \right) - u\left( a
\right)} \right],\,\,\,\,\,\,\,\,\,\,\,\,\,\,\,\,\,\,\,\,\,\,\,\,\,\,\,\,\,\,\,\,\,\,\,\,\,\,\,\,\,\,\,\,\,\,\,\,\,\,\,\,\,\,\,\,\,\,\,\,\,\,\,\,\,\,\,\,\,\,\,\,\,\, (1),(4^{\prime}), \,\,\left( {[11]} \right) \\
\\
 H_u\left[ {\frac{{b - a}}{2} + \left| {x - \frac{{a + b}}{2}}
\right|} \right]^s \cdot \left[ {f\left( b \right) - f\left( a
\right)} \right],\,\,\,\,\,\,\,\,\,\,\,\,\,\,\,\,\,\,\,\,\,\,\,\,\,\,\,\,\,\,\,\,\,\,\,\,\,\,\,\,\,\,\,\,\,\,\,\,\,\,\,\,\,\,\,\,\,\,\,\,\,\,\,\,\,\,\,\,\,\,\,\,\,\, (1^{\prime}), (4), \,\,\left( {[11]} \right) \\
\\
\mathop {\sup }\limits_{t \in \left[ {a,x} \right]} \left\{
{\left( {x - t} \right)\mu \left( {f;x,t} \right)} \right\} \cdot
\bigvee_a^x \left( u \right)  + \mathop {\sup }\limits_{t \in
\left[ {x,b} \right]} \left\{ {\left( {t - x} \right)\mu \left(
{f;x,t} \right)} \right\} \cdot \bigvee_x^b \left( u \right),\,\,\,\,\, (2^{\prime}), (7), \,\,\left( {[7]} \right) \\
\\
\frac{1}{2}\left[ {\left( {x - a} \right) \cdot \bigvee_a^x \left(
u \right) \cdot \left\| {f'} \right\|_{\infty ,\left[ {a,x}
\right]} + \left( {b - x} \right) \cdot \bigvee_x^b \left( u
\right) \cdot \left\| {f'} \right\|_{\infty ,\left[ {x,b} \right]}
} \right]
\\
+ \frac{1}{2}\left| {f'\left( x \right)} \right| \cdot \left[
{\left( {x - a} \right) \cdot \bigvee_a^x \left( u \right) +
\left( {b -
x} \right) \cdot \bigvee_x^b \left( u \right)} \right],\,\,\,\,\,\,\,\,\,\,\,\,\,\,\,\,\,\,\,\,\,\,\,\,\,\,\,\,\,\,\,\,\,\,\,\,\,\,\,\,\,\,\,\,\,\,\ (2^{\prime}), (7),(8), \,\,\left( {[7]} \right) \\\\
 \end{array} \right.
\end{multline}
More details about each inequality of the above, the reader may
refer to the corresponding mentioned references and the references
therein.

From a different view point, the authors of \cite{Dragomir5}
considered the problem of approximating the Stieltjes integral
$\int_a^b {f\left( t \right)du\left( t \right)}$ via the
generalized trapezoid rule $\left[ {u\left( x \right) - u\left( a
\right)} \right]f\left( a \right) + \left[ {u\left( b \right) -
u\left( x \right)} \right]f\left( b \right)$
\begin{align*}
\mathcal{T}\left( f,u;x \right) := \left[ {u\left( x \right) -
u\left( a \right)} \right]f\left( a \right) + \left[ {u\left( b
\right) - u\left( x \right)} \right]f\left( b \right) - \int_a^b
{f\left( t \right)du\left( t \right)}
\end{align*}

\begin{multline}
\label{eq.Dra1} \left| {\mathcal{T}\left( f,u;x \right)} \right|
\\
\le \left\{ \begin{array}{l}
 H_u\left[ {\frac{{b - a}}{2} + \left| {x - \frac{{a + b}}{2}}
\right|} \right]^r \cdot \bigvee_a^b\left( {f} \right),\,\,\,\,\,\,\,\,\,\,\,\,\,\,\,\,\,\,\,\,\,\,\,\,\,\,\,\,\,\,\,\,\,\,\,\,\,\,\,\,\,\,\,\,\,\,\,\,\,\,\,\,\,\,\,\,\,\,\,\,\,\,\,\,\,\,\,\,\,\,\,\,\,\,\,\,\,\,\,\,\,\,\,\,\,\,\,\,\,\,\, (1^{\prime}),(2), \,\,\left( {[15]} \right) \\
  \\
 H_f \left\{ \begin{array}{l} \left[ {\left( {x - a} \right)^s +
\left( {b - x} \right)^s } \right]\left[ {\frac{1}{2}\bigvee_a^b
\left( u \right) + \frac{1}{2}\left| {\bigvee_a^x \left( u \right)
- \bigvee_x^b \left( u
\right)} \right|} \right] \\
\\
\left[ { \left( {x - a} \right)^{qs}  + \left( {b - x}
\right)^{qs} } \right]^{1/q} \left[ {\left( {\bigvee_{a}^x\left( u
\right)} \right)^p  + \left(
{\bigvee_{x}^b\left( u \right)} \right)^p } \right]^{1/p},\,\,\,\,\,\,\,\,\,\,\,\,\,\,\, (1),(2^{\prime}),\,\, \,\,\left( {[8]} \right)\\
\,\,\,\,\,\,\,\,\,\,\,\,\,\,\,\,\,\,\,\,\,\,\,\,\,\,\,\,\,\,\,\,\,\,\,\,\,\,\,\,\,\,\,\,\,\,\,\,\,\,\,\,\,\,\,\,\,\,\,\,\,\,\,\,\,\,\,\,\,\,\,\,\,\,\,\,\,\,\,\,\,\,\,\,\,\,\,\,\,\,\,\,p>1\,\frac{1}{p}+\frac{1}{q}=1\\
 \left[ {\frac{{b - a}}{2} + \left| {x - \frac{{a + b}}{2}}
\right|} \right]^s  \cdot \bigvee_a^b \left( {u} \right) \\
 \end{array} \right. \\
 \end{array} \right.
\end{multline}
For new quadrature rules involving $\mathcal{RS}$--integral see
the recent works \cite{alomari1}--\cite{alomari3}. For other
results concerning various approximation for
$\mathcal{RS}$--integral under various assumptions on $f$ and $u$,
see \cite{Anastassiou,Barnett,CeroneDragomir1,Cerone},
\cite{Dragomir6}--\cite{Dragomir8} and the references therein.\\

In the recent work \cite{Liu}, Z. Liu has proved sharp
generalization of weighted Ostrowski type inequality for mappings
of bounded variation, as follows (see also \cite{WLiu}):
\begin{theorem}\label{thm.liu}
Let $f:[a,b] \to \mathbb{R}$  be a mapping of bounded variation,
$g: [a,b] \to [0, \infty)$ continuous and positive on $(a,b)$.
Then for any $x \in [a,b]$ and $\alpha \in [0,1]$, we have
\begin{multline}
\label{liu.ineq} \left| {\int_a^b {f\left( t \right)g\left( t
\right)dt}  - \left[ {\left( {1 - \alpha } \right)f\left( x
\right)\int_a^b {g\left( t \right)dt}}\right.}\right.
\\
\left. {+\left. {\alpha \left( {f\left( a \right)\int_a^x {g\left(
t \right)dt}  + f\left( b \right)\int_x^b {g\left( t \right)dt} }
\right)} \right]} \right|
\\
\le \left[ {\frac{1}{2} + \left| {\frac{1}{2} - \alpha } \right|}
\right]\left[ {\frac{1}{2}\int_a^b {g\left( t \right)dt}  + \left|
{\int_a^x {g\left( t \right)dt}  - \frac{1}{2}\int_a^b {g\left( t
\right)dt} } \right|} \right] \cdot \bigvee_a^b \left( f \right)
\end{multline}
where, $\bigvee_a^b\left( {f} \right)$ denotes to the total
variation of $f$ over $[a,b]$. The constant $\left[ {\frac{1}{2} +
\left| {\frac{1}{2} - \alpha } \right|} \right]$ is the best
possible.
\end{theorem}
For recent results concerning Ostrowski inequality for mappings of
bounded variation see
\cite{Dragomir1},\cite{Liu}--\cite{Tseng3}.\\

The main aim in this paper, is to introduce and discuss new
weighted generalizations of the Ostrowski and the generalized
trapezoid inequalities for the Riemann--Stieltjes integrals.

\section{The Results}

We begin with the following result:
\begin{theorem}\label{thm1}
Let $g,u: [a,b] \to [0, \infty)$ be such that $g$ is continuous
and positive on $[a,b]$ and $u$ is monotonic increasing on
$[a,b]$. If $f:[a,b] \to \mathbb{R}$ is a mapping of bounded
variation on $[a,b]$, then for any $x \in [a,b]$ and $\alpha \in
[0,1]$, we have
\begin{multline}
\label{eq2.1}\left| {\left( {1 - \alpha } \right)\left[ {f\left( x
\right)\int_a^{{\textstyle{{a + b} \over 2}}} {g\left( s
\right)du\left( s \right)}  + f\left( {a + b - x}
\right)\int_{{\textstyle{{a + b} \over 2}}}^b {g\left( s
\right)du\left( s \right)} } \right]} \right.
\\
\left. {+ \alpha \left[ {f\left( a \right)\int_a^x {g\left( s
\right)du\left( s \right)}  + f\left( b \right)\int_x^b {g\left( s
\right)du\left( s \right)} } \right] - \int_a^b {f\left( t
\right)g\left( t \right)du\left( t \right)}}\right|
\\
\le \left[ {\frac{1}{2} + \left| {\frac{1}{2} - \alpha } \right|}
\right]\left[ {\frac{1}{2}\int_a^b {g\left( t \right)du\left( t
\right)}  + \left| {\int_a^x {g\left( t \right)du\left( t \right)}
- \frac{1}{2}\int_a^b {g\left( t \right)du\left( t \right)} }
\right|} \right] \cdot \bigvee_a^b \left( f \right)
\end{multline}
where, $\bigvee_a^b\left( {f} \right)$ denotes to the total
variation of $f$ over $[a,b]$. The constant $\left[ {\frac{1}{2} +
\left| {\frac{1}{2} - \alpha } \right|} \right]$ is the best
possible.
\end{theorem}

\begin{proof}
Define the mapping
\begin{align*}
K_{g,u} \left( {t;x} \right): = \left\{ \begin{array}{l}
 \left( {1 - \alpha } \right)\int_a^t {g\left( s \right)du\left( s \right)}  + \alpha \int_x^t {g\left( s \right)du\left( s \right)} ,\,\,\,\,\,\,\,\,\,\,\,\,\,\,t \in \left[ {a,x} \right] \\
  \\
 \left( {1 - \alpha } \right)\int_{{\textstyle{{a + b} \over 2}}}^t {g\left( s \right)du\left( s \right)}  + \alpha \int_x^t {g\left( s \right)du\left( s \right)} ,\,\,\,\,\,\,\,t \in \left( {x,a + b - x} \right] \\
  \\
 \left( {1 - \alpha } \right)\int_b^t {g\left( s \right)du\left( s \right)}  + \alpha \int_x^t {g\left( s \right)du\left( s \right)} ,\,\,\,\,\,\,\,\,\,\,\,\,\,\,t \in \left( {a + b - x,b} \right] \\
 \end{array} \right.
\end{align*}
Using integration by parts, we have the following identity
\begin{align*}
\int_a^b {K_{g,u} \left( {t;x} \right)df\left( t \right)}  &=
\int_a^x {\left[ {\left( {1 - \alpha } \right)\int_a^t {g\left( s
\right)du\left( s \right)}  + \alpha \int_x^t {g\left( s
\right)du\left( s \right)} } \right]df\left( t \right)}
\\
&\qquad+ \int_x^{a + b - x} {\left[ {\left( {1 - \alpha }
\right)\int_{{\textstyle{{a + b} \over 2}}}^t {g\left( s
\right)du\left( s \right)}  + \alpha \int_x^t {g\left( s
\right)du\left( s \right)} } \right]df\left( t \right)}
\\
&\qquad + \int_{a + b - x}^b {\left[ {\left( {1 - \alpha }
\right)\int_b^t {g\left( s \right)du\left( s \right)}  + \alpha
\int_x^t {g\left( s \right)du\left( s \right)} } \right]df\left( t
\right)}
\\
&= \left( {1 - \alpha } \right)\left[ {f\left( x
\right)\int_a^{{\textstyle{{a + b} \over 2}}} {g\left( s
\right)du\left( s \right)}  + f\left( {a + b - x}
\right)\int_{{\textstyle{{a + b} \over 2}}}^b {g\left( s
\right)du\left( s \right)} } \right]
\\
&\qquad+ \alpha \left[ {f\left( a \right)\int_a^x {g\left( s
\right)du\left( s \right)}  + f\left( b \right)\int_x^b {g\left( s
\right)du\left( s \right)} } \right] - \int_a^b {f\left( t
\right)g\left( t \right)du\left( t \right)}
\end{align*}
Using the fact that for a continuous function $p:[a,b] \to
\mathbb{R}$ and a function $\nu:[a,b] \to \mathbb{R}$ of bounded
variation, then the Riemann--Stieltjes integral $\int_a^b {p\left(
t \right)d\nu\left( t \right)}$ exists and one has the inequality
\begin{align}
\label{keyeq1}\left| {\int_a^b {p\left( t \right)d\nu\left( t
\right)} } \right| \le \mathop {\sup }\limits_{t \in \left[ {a,b}
\right]} \left| {p\left( t \right)} \right| \bigvee_a^b\left( \nu
\right).
\end{align}
As $f$ is of bounded variation on $[a,b]$, by (\ref{keyeq1}) we
have
\begin{multline}
\label{eq2.3}\left| {\left( {1 - \alpha } \right)\left[ {f\left( x
\right)\int_a^{{\textstyle{{a + b} \over 2}}} {g\left( s
\right)du\left( s \right)}  + f\left( {a + b - x}
\right)\int_{{\textstyle{{a + b} \over 2}}}^b {g\left( s
\right)du\left( s \right)} } \right]} \right.
\\
\left. {+ \alpha \left[ {f\left( a \right)\int_a^x {g\left( s
\right)du\left( s \right)}  + f\left( b \right)\int_x^b {g\left( s
\right)du\left( s \right)} } \right] - \int_a^b {f\left( t
\right)g\left( t \right)du\left( t \right)}}\right|
\\
\le \mathop {\sup }\limits_{t \in \left[ {a,b} \right]} \left|
{K_{g,u} \left( {t;x} \right)} \right| \cdot \bigvee_a^b \left( f
\right).
\end{multline}
Now, define the mappings $p,q:[a,b] \to \mathbb{R}$ given by
\begin{align*}
p_1\left( t \right)&:= \left( {1 - \alpha } \right)\int_a^t
{g\left( s \right)du\left( s \right)}  + \alpha \int_x^t {g\left(
s \right)du\left( s \right)}
,\,\,\,\,\,\,\,\,\,\,\,\,\,\,\,\,\,\,\,\,t \in \left[ {a,x}
\right],
\\
p_2\left( t \right)&:=  \left( {1 - \alpha }
\right)\int_{{\textstyle{{a + b} \over 2}}}^t {g\left( s
\right)du\left( s \right)}  + \alpha \int_x^t {g\left( s
\right)du\left( s \right)} ,\,\,\,\,\,\,\,\,\,\,\,\,\,\,\,t \in
\left( {x,a + b - x} \right]
\\
p_3\left( t \right) &:= \left( {1 - \alpha } \right)\int_b^t
{g\left( s \right)du\left( s \right)}  + \alpha \int_x^t {g\left(
s \right)du\left( s \right)}
,\,\,\,\,\,\,\,\,\,\,\,\,\,\,\,\,\,\,\,\,t \in \left( {a + b -
x,b} \right]
\end{align*}
for all $\alpha \in [0,1]$, and $x \in [a,b]$. Since $g$ is
\emph{positive} continuous and $u$ is monotonic increasing on
$[a,b]$ then the Riemann--Stieltjes integral $\int_{a}^{b}
{g\left( s \right)du\left( s \right)}$ exists and \emph{positive}.
Also, since the derivative of the monotonic increasing function
$u$ is always positive, so that $\left( gu^{\prime} \right)\left(
t \right)>0$ a.e., it follows that, $p_1^{\prime}\left( t \right),
p_2^{\prime}\left( t \right), p_3^{\prime}\left( t \right)
>0$, almost everywhere on their corresponding domains. Therefore, we have
\begin{align*}
\mathop {\sup }\limits_{t \in \left[ {a,x} \right]} \left|
{K_{g,u} \left( {t;x} \right)} \right| &= \max \left\{ {\left( {1
- \alpha } \right)\int_a^x {g\left( s \right)du\left( s \right)}
,\alpha \int_a^x {g\left( s \right)du\left( s \right)} } \right\}
\\
&= \left[ {\frac{1}{2} + \left| {\frac{1}{2} - \alpha } \right|}
\right] \cdot \int_a^x {g\left( s \right)du\left( s \right)},
\end{align*}
\begin{align*}
&\mathop {\sup }\limits_{t \in \left( {x,a+b-x} \right]} \left|
{K_{g,u} \left( {t;x} \right)} \right|
\\
&= \max \left\{ {\left( {1 - \alpha }
\right)\int_x^{{\textstyle{{a + b} \over 2}}} {g\left( s
\right)du\left( s \right)} ,\alpha \int_x^{{\textstyle{{a + b}
\over 2}}} {g\left( s \right)du\left( s \right)}  +
\int_{{\textstyle{{a + b} \over 2}}}^{a + b - x} {g\left( s
\right)du\left( s \right)}} \right\}
\\
&= \frac{1}{2}\left[ {\int_x^{a + b - x} {g\left( s
\right)du\left( s \right)}  + \left( {1 - \alpha } \right)\left|
{\int_{{\textstyle{{a + b} \over 2}}}^{a + b - x} {g\left( s
\right)du\left( s \right)} } \right|} \right],
\end{align*}
and
\begin{align*}
\mathop {\sup }\limits_{t \in \left( {a+b-x,b} \right]} \left|
{K_{g,u} \left( {t;x} \right)} \right| &= \max \left\{ {\left( {1
- \alpha } \right)\int_x^b {g\left( s \right)du\left( s \right)}
,\alpha \int_x^b {g\left( s \right)du\left( s \right)} } \right\}
\\
&= \left[ {\frac{1}{2} + \left| {\frac{1}{2} - \alpha } \right|}
\right] \cdot \int_x^b {g\left( s \right)du\left( s \right)}.
\end{align*}
Thus
\begin{align}
\label{eq2.4}\mathop {\sup }\limits_{t \in \left[ {a,b} \right]}
\left| {K_{g,u} \left( {t;x} \right)} \right| &= \left[
{\frac{1}{2} + \left| {\frac{1}{2} - \alpha } \right|} \right]
\cdot \max \left\{ {\int_a^x {g\left( s \right)du\left( s \right)}
,\int_x^b {g\left( s \right)du\left( s \right)} } \right\}
\\
&= \left[ {\frac{1}{2} + \left| {\frac{1}{2} - \alpha } \right|}
\right] \cdot \left[ {\frac{1}{2}\int_a^b {g\left( s
\right)du\left( s \right)} + \left| {\int_a^x {g\left( s
\right)du\left( s \right)} - \frac{1}{2}\int_a^b {g\left( s
\right)du\left( s \right)} } \right|} \right].\nonumber
\end{align}
Therefore, by (\ref{eq2.3}) and (\ref{eq2.4}) we get
(\ref{eq2.1}). To prove that the constant $\frac{1}{2} + \left|
{\frac{1}{2} - \alpha } \right|$ is best possible for all $\alpha
\in [0,1]$, take $u(t)= t$ for all $t \in [a,b]$ and therefore, we
refer to (\ref{liu.ineq}). Thus, the sharpness follows from
(\ref{liu.ineq}), (consider $f$ and $g$ to be defined as in
\cite{Liu}). Hence, the proof is established and we shall omit the
details.
\end{proof}

\begin{corollary}
In Theorem \ref{thm1}, choose $\alpha = 0$, then we get
\begin{multline}
\left| {f\left( x
\right)\int_a^{{\textstyle{{a + b} \over 2}}} {g\left( s
\right)du\left( s \right)}  + f\left( {a + b - x}
\right)\int_{{\textstyle{{a + b} \over 2}}}^b {g\left( s
\right)du\left( s \right)}  - \int_a^b {f\left( t
\right)g\left( t \right)du\left( t \right)}}\right|
\\
\le \left[ {\frac{1}{2}\int_a^b {g\left( t \right)du\left( t
\right)}  + \left| {\int_a^x {g\left( t \right)du\left( t \right)}
- \frac{1}{2}\int_a^b {g\left( t \right)du\left( t \right)} }
\right|} \right] \cdot \bigvee_a^b \left( f \right).
\end{multline}
A general weighted version of the above Ostrowski inequality for
$\mathcal{RS}$ integrals, may be deduced as follows:
\begin{align}
\left| {f\left( x \right) -\frac{\int_a^b {f\left( t
\right)g\left( t \right)du\left( t \right)}}{\int_a^b {g\left( t
\right)du\left( t \right)}} }\right| \le \left[ {\frac{1}{2}  +
\left| {\frac{\int_a^x {g\left( t \right)du\left( t
\right)}}{\int_a^b {g\left( t \right)du\left( t \right)}} -
\frac{1}{2}} \right|} \right] \cdot \bigvee_a^b \left( f \right)
\end{align}
provided that $g(t)\ge 0$, for almost every $t\in [a,b]$ and
$\int_a^b {g\left( t \right)du\left( t \right)} \ne 0$.
\end{corollary}

\begin{remark}
Choosing $\alpha = 1$ in (\ref{eq2.1}), then we get
\begin{multline}
\left| {f\left( a \right)\int_a^x {g\left( s \right)du\left( s
\right)} + f\left( b \right)\int_x^b {g\left( s \right)du\left( s
\right)}- \int_a^b {f\left( t \right)g\left( t \right)du\left( t
\right)} } \right|
\\
\le \left[ {\frac{1}{2}\int_a^b {g\left( t \right)du\left( t
\right)}  + \left| {\int_a^x {g\left( t \right)du\left( t \right)}
- \frac{1}{2}\int_a^b {g\left( t \right)du\left( t \right)} }
\right|} \right] \cdot \bigvee_a^b \left( f \right),
\end{multline}
which is  `the generalized trapezoid inequality for
$\mathcal{RS}$--integrals'
\end{remark}

\begin{corollary}
\label{cor1}In Theorem \ref{thm1}, let $g\left( t \right) = 1$ for
all $t \in [a,b]$. Then, we have the inequality
\begin{multline}
\label{eq2.8}\left| {\alpha \left[ {\left( {u\left( x \right) -
u\left( a \right)} \right)f\left( a \right) + \left( {u\left( b
\right) - u\left( x \right)} \right)f\left( b \right)} \right] + \left( {1 - \alpha } \right)\left\{ {\left[ {u\left( {\frac{{a + b}}{2}} \right) - u\left( a \right)} \right]f\left( x \right)} \right.}
\right.
\\
\left. {+ \left. {\left[ {u\left( b \right) - u\left( {\frac{{a + b}}{2}} \right)} \right]f\left( {a + b - x} \right)} \right\}
- \int_a^b {f\left( t
\right)du\left( t \right)}}\right|
\\
\le \left[ {\frac{1}{2} + \left| {\frac{1}{2} - \alpha } \right|}
\right] \cdot\left[ {\frac{{u\left( b \right) - u\left( a
\right)}}{2} + \left| {u\left( x \right) - \frac{{u\left( a
\right) + u\left( b \right)}}{2}} \right|} \right] \cdot
\bigvee_a^b \left( f \right).
\end{multline}
The constant $\left[ {\frac{1}{2} + \left| {\frac{1}{2} - \alpha }
\right|} \right]$ is the best possible.

For instance,
\begin{itemize}
\item If $\alpha = 0$, then we get
\begin{multline}
\left| {\left\{ {\left[ {u\left( {\frac{{a + b}}{2}} \right) - u\left( a \right)} \right]f\left( x \right)
+ \left[ {u\left( b \right) - u\left( {\frac{{a + b}}{2}} \right)} \right]f\left( {a + b - x} \right)} \right\}
- \int_a^b {f\left( t
\right)du\left( t \right)}}\right|
\\
\le \left[ {\frac{{u\left( b \right) - u\left( a \right)}}{2} +
\left| {u\left( x \right) - \frac{{u\left( a \right) + u\left( b
\right)}}{2}} \right|} \right] \cdot \bigvee_a^b \left( f \right).
\end{multline}

\item If $\alpha = \frac{1}{3}$, then we get
\begin{multline}
\label{eq2.10}\left| {\frac{1}{3} \left[ {\left( {u\left( x \right) -
u\left( a \right)} \right)f\left( a \right) + \left( {u\left( b
\right) - u\left( x \right)} \right)f\left( b \right)} \right] + \frac{2}{3}\left\{ {\left[ {u\left( {\frac{{a + b}}{2}} \right) - u\left( a \right)} \right]f\left( x \right)} \right.}
\right.
\\
\left. {+ \left. {\left[ {u\left( b \right) - u\left( {\frac{{a + b}}{2}} \right)} \right]f\left( {a + b - x} \right)} \right\}
- \int_a^b {f\left( t
\right)du\left( t \right)}}\right|
\\
\le \frac{2}{3}\left[ {\frac{{u\left( b \right) - u\left( a
\right)}}{2} + \left| {u\left( x \right) - \frac{{u\left( a
\right) + u\left( b \right)}}{2}} \right|} \right] \cdot
\bigvee_a^b \left( f \right).
\end{multline}

\item If $\alpha = \frac{1}{2}$, then we get
\begin{multline}
\left| {\frac{1}{2}\left\{ {\left( {u\left( x \right) -
u\left( a \right)} \right)f\left( a \right) + \left( {u\left( b
\right) - u\left( x \right)} \right)f\left( b \right)+ \left[ {u\left( {\frac{{a + b}}{2}} \right) - u\left( a \right)} \right]f\left( x \right)} \right.}
\right.
\\
\left. {+ \left. {\left[ {u\left( b \right) - u\left( {\frac{{a + b}}{2}} \right)} \right]f\left( {a + b - x} \right)} \right\}
- \int_a^b {f\left( t
\right)du\left( t \right)}}\right|
\\
\le \frac{1}{2}\left[ {\frac{{u\left( b \right) - u\left( a
\right)}}{2} + \left| {u\left( x \right) - \frac{{u\left( a
\right) + u\left( b \right)}}{2}} \right|} \right] \cdot
\bigvee_a^b \left( f \right).
\end{multline}

\item If $\alpha = 1$, then we get
\begin{multline}
\left| {\left[ {u\left( x \right) - u\left( a \right)}
\right]f\left( a \right) + \left[ {u\left( b \right) - u\left( x
\right)} \right]f\left( b \right)- \int_a^b {f\left( t
\right)du\left( t \right)}}\right|
\\
\le \left[ {\frac{{u\left( b \right) - u\left( a \right)}}{2} +
\left| {u\left( x \right) - \frac{{u\left( a \right) + u\left( b
\right)}}{2}} \right|} \right] \cdot \bigvee_a^b \left( f \right).
\end{multline}
\end{itemize}
\end{corollary}
\begin{proof}
The results follow by Theorem \ref{thm1}. It remains to prove the
sharpness of (\ref{eq2.8}). Suppose $0 \le \alpha \le
\frac{1}{2}$, assume that (\ref{eq2.8}) holds with constant
$C_1>0$, i.e.,
\begin{multline}
\label{eqK}\left| {\alpha \left[ {\left( {u\left( x \right) -
u\left( a \right)} \right)f\left( a \right) + \left( {u\left( b
\right) - u\left( x \right)} \right)f\left( b \right)} \right] + \left( {1 - \alpha } \right)\left\{ {\left[ {u\left( {\frac{{a + b}}{2}} \right) - u\left( a \right)} \right]f\left( x \right)} \right.}
\right.
\\
\left. {+ \left. {\left[ {u\left( b \right) - u\left( {\frac{{a + b}}{2}} \right)} \right]f\left( {a + b - x} \right)} \right\}
- \int_a^b {f\left( t
\right)du\left( t \right)}}\right|
\\
\le C_1 \cdot\left[ {\frac{{u\left( b \right) - u\left( a
\right)}}{2} + \left| {u\left( x \right) - \frac{{u\left( a
\right) + u\left( b \right)}}{2}} \right|} \right] \cdot
\bigvee_a^b \left( f \right).
\end{multline}
Let $f,u:[a,b] \to \mathbb{R}$ be defined as follows $u\left( t
\right) = t$ and
\begin{align*}
f\left( t \right) = \left\{
\begin{array}{l}
 0,\,\,\,\,\,\,\,\,t \in \left[ {a,b} \right]\backslash \left\{ {\frac{{a + b}}{2}} \right\} \\
  \\
 \frac{1}{2},\,\,\,\,\,\,\,t = \frac{{a + b}}{2} \\
 \end{array} \right.,
\end{align*}
which follows that $\bigvee_a^b\left( f \right) = 1$ and $\int_a^b
{f\left( t \right)du\left( t \right)}  = 0$, setting $x =
\frac{a+b}{2}$ it gives by (\ref{eqK})
\begin{align*}
\left( {1 - \alpha } \right)\frac{\left( {b-a} \right)}{2} \le
C_1\frac{\left( {b-a} \right)}{2}.
\end{align*}
which proves that $C_1 \ge 1 - \alpha$, and therefore $1- \alpha$
is the best possible for all $0 \le \alpha \le \frac{1}{2}$.

Now, suppose $\frac{1}{2} \le \alpha \le 1$ and assume that
(\ref{eq2.8}) holds with constant $C_2>0$, i.e.,
\begin{multline}
\label{eqKK}\left| {\alpha \left[ {\left( {u\left( x \right) -
u\left( a \right)} \right)f\left( a \right) + \left( {u\left( b
\right) - u\left( x \right)} \right)f\left( b \right)} \right] + \left( {1 - \alpha } \right)\left\{ {\left[ {u\left( {\frac{{a + b}}{2}} \right) - u\left( a \right)} \right]f\left( x \right)} \right.}
\right.
\\
\left. {+ \left. {\left[ {u\left( b \right) - u\left( {\frac{{a + b}}{2}} \right)} \right]f\left( {a + b - x} \right)} \right\}
- \int_a^b {f\left( t
\right)du\left( t \right)}}\right|
\\
\le C_2 \cdot\left[ {\frac{{u\left( b \right) - u\left( a
\right)}}{2} + \left| {u\left( x \right) - \frac{{u\left( a
\right) + u\left( b \right)}}{2}} \right|} \right] \cdot
\bigvee_a^b \left( f \right).
\end{multline}
Let $f,u:[a,b] \to \mathbb{R}$ be defined as follows $u\left( t
\right) = t$ and
\begin{align*}
f\left( t \right) = \left\{
\begin{array}{l}
 0,\,\,\,\,\,\,\,\,t \in \left( {a,b} \right]\\
  \\
 1,\,\,\,\,\,\,\,\,t = a \\
 \end{array} \right.,
\end{align*}
which follows that $\bigvee_a^b\left( f \right) = 1$ and $\int_a^b
{f\left( t \right)du\left( t \right)}  = 0$, setting $x =
\frac{a+b}{2}$ it gives by (\ref{eqKK})
\begin{align*}
\alpha \frac{\left( {b-a} \right)}{2} \le C_2\frac{\left( {b-a}
\right)}{2}.
\end{align*}
which proves that $C_2 \ge \alpha$, and therefore $\alpha$ is the
best possible for all $\frac{1}{2} \le \alpha \le 1$.
Consequently, we can conclude that the constant $\left[
{\frac{1}{2} + \left| {\frac{1}{2} - \alpha } \right|} \right]$ is
the best possible, for all $\alpha\in [0,1]$.
\end{proof}

\begin{corollary}
In (\ref{eq2.16}), setting $x = \frac{a+b}{2}$ then we have the
following Simpson type inequality for Riemann--Stieltjes integral:
\begin{multline}
\left| {\frac{1}{3}\left\{ {\left[ {u\left( {\frac{a+b}{2}}
\right) - u\left( a \right)} \right]f\left( a \right) + 2 \left[
{u\left( b \right) - u\left( a \right)} \right]f\left(
{\frac{a+b}{2}} \right) } \right.} \right.
\\
\left. {\left. {+ \left[ {u\left( b \right) - u\left(
{\frac{a+b}{2}} \right)} \right]f\left( b \right)} \right\} -
\int_a^b {f\left( t \right)du\left( t \right)}}\right|
\\
\le \frac{2}{3}\left[ {\frac{{u\left( b \right) - u\left( a
\right)}}{2} + \left| {u\left( {\frac{a+b}{2}} \right) -
\frac{{u\left( a \right) + u\left( b \right)}}{2}} \right|}
\right] \cdot \bigvee_a^b \left( f \right).
\end{multline}
The constant $\frac{2}{3}$ is the best possible.
\end{corollary}

\begin{remark}
For recent three-point quadrature rules and related inequalities
regarding Riemann--Stieltjes integrals, the reader may refer to
the work \cite{alomari3}.
\end{remark}

\begin{corollary}
In (\ref{eq2.8}), let $u(t) = t$ for all $t \in [a,b]$, then we
get
\begin{multline}
\left| {\alpha \left( {\left( {x - a} \right)f\left( a \right) +
\left( {b - x} \right)f\left( b \right)} \right) +
\frac{1}{2}\left( {1 - \alpha } \right)\left( {b - a} \right)\left( {f\left( x \right) + f\left( {a + b - x} \right)} \right)
- \int_a^b
{f\left( t \right)dt}}\right|
\\
\le \left[ {\frac{1}{2} + \left| {\frac{1}{2} - \alpha } \right|}
\right] \cdot\left[ {\frac{{b - a}}{2} + \left| {x - \frac{{a +
b}}{2}} \right|} \right] \cdot \bigvee_a^b \left( f \right).
\end{multline}
For $x = \frac{a+b}{2}$, we have
\begin{multline}
\left| {\left( {b - a} \right) \left[ {\alpha  \frac{f\left( a
\right) + f\left( b \right)}{2} + \left( {1 - \alpha }
\right)f\left( {\frac{a+b}{2}} \right)} \right]- \int_a^b {f\left(
t \right)dt}}\right|
\\
\le \left[ {\frac{1}{2} + \left| {\frac{1}{2} - \alpha } \right|}
\right] \cdot\frac{\left( {b-a} \right)}{2} \cdot \bigvee_a^b
\left( f \right).
\end{multline}
\end{corollary}

\begin{remark}
Under the assumptions of Theorem \ref{thm1}, a weighted
generalization of Montgomery's type identity for
Riemann--Stieltjes integrals may be deduced as follows:
\begin{align*}
f\left( x \right) = \frac{1}{\int_a^b {g\left( s \right)du\left( s
\right)}}\int_a^b {K_{g,u} \left( {t;x} \right)df\left( t \right)}
+ \frac{1}{\int_a^b {g\left( s \right)du\left( s \right)}}\int_a^b
{f\left( t \right)g\left( t \right)du\left( t \right)},
\end{align*}
for all $x \in \left[ {a,b} \right]$, where
\begin{align*}
K_{g,u} \left( {t;x} \right): = \left\{ \begin{array}{l}
 \int_a^t {g\left( s \right)du\left( s \right)} ,\,\,\,\,\,\,\,t \in \left[ {a,x} \right] \\
  \\
 \int_b^t {g\left( s \right)du\left( s \right)} ,\,\,\,\,\,\,\,t \in \left( {x,b} \right] \\
 \end{array} \right..
\end{align*}
Provided that $\int_a^b {g\left( s \right)du\left( s \right)} \ne
0$.
\end{remark}

\section{On $L$-Lipschitz integrators}

\begin{theorem}\label{thm2}
Let $g$ be as in Theorem \ref{thm1}. Let $u: [a,b] \to [0,
\infty)$ be of bounded variation on $[a,b]$. If $f:[a,b] \to
\mathbb{R}$ is $L$--Lipschitzian on $[a,b]$, then for any $x \in
[a,b]$ and $\alpha \in [0,1]$, we have
\begin{multline}
\label{eq2.13}\left| {\left( {1 - \alpha } \right)\left[ {f\left( x
\right)\int_a^{{\textstyle{{a + b} \over 2}}} {g\left( s
\right)du\left( s \right)}  + f\left( {a + b - x}
\right)\int_{{\textstyle{{a + b} \over 2}}}^b {g\left( s
\right)du\left( s \right)} } \right]} \right.
\\
\left. {+ \alpha \left[ {f\left( a \right)\int_a^x {g\left( s
\right)du\left( s \right)}  + f\left( b \right)\int_x^b {g\left( s
\right)du\left( s \right)} } \right] - \int_a^b {f\left( t
\right)g\left( t \right)du\left( t \right)}}\right|
\\
\le L \cdot \max \left\{ {\left( {x - a} \right) \cdot\mathop
{\sup }\limits_{t \in \left[ {a,x} \right]} \left\{ {M\left( t
\right)} \right\},\left( {b - x} \right) \cdot\mathop {\sup
}\limits_{t \in \left[ {x,b} \right]} \left\{ {N\left( t \right)}
\right\}} \right\} \cdot \bigvee_a^b \left( u \right)
\end{multline}
where,
\begin{align*}
M\left( t \right):=\max \left\{ {\left( {1 - \alpha }
\right)\mathop {\sup }\limits_{s \in \left[ {a,t} \right]} \left|
{g\left( s \right)} \right|,\alpha \mathop {\sup }\limits_{s \in
\left[ {t,x} \right]} \left| {g\left( s \right)} \right|}
\right\},
\end{align*}
and
\begin{align*}
N\left( t \right):= \max \left\{ {\left( {1 - \alpha }
\right)\mathop {\sup }\limits_{s \in \left[ {t,b} \right]} \left|
{g\left( s \right)} \right|,\alpha \mathop {\sup }\limits_{s \in
\left[ {t,x} \right]} \left| {g\left( s \right)} \right|}
\right\}.
\end{align*}
\end{theorem}

\begin{proof}
By Theorem \ref{thm1}, we have the identity
\begin{align*}
&\alpha \left[ {f\left( a \right)\int_a^x {g\left( s
\right)du\left( s \right)}  + f\left( b \right)\int_x^b {g\left( s
\right)du\left( s \right)} } \right]+ \left( {1 - \alpha }
\right)f\left( x \right)\int_a^b {g\left( s \right)du\left( s
\right)}
\\
&\qquad  - \int_a^b {f\left( t \right)g\left( t \right)du\left( t
\right)}
\\
&=\int_a^b {K_{g,u} \left( {t;x} \right)df\left( t \right)}.
\end{align*}
Using the fact that for a Riemann integrable function $p:[c,d] \to
\mathbb{R}$ and $L$-Lipschitzian function $\nu:[c,d] \to
\mathbb{R}$, the inequality one has the inequality
\begin{align}
\label{keyeq2}\left| {\int_c^d {p\left( t \right)d\nu\left( t
\right)} } \right| \le L \int_c^d {\left| {p\left( t \right)}
\right|dt}.
\end{align}
As $f$ is $L$--Lipschitz mapping on $[a,b]$, by (\ref{keyeq2}) we
have
\begin{align}
\label{eq2.15}\left| {\int_a^b {K_{g,u} \left( {t;x}
\right)df\left( t \right)}}\right| &\le L \int_a^b {\left|
{K_{g,u} \left( {t;x} \right)}\right|dt}
\nonumber\\
&= L \left[ {\int_a^x {\left| {p\left( {t} \right)}\right|dt} +
\int_x^b {\left| {q\left( {t} \right)}\right|dt}}\right]
\end{align}
However, as $u$ is of bounded variation on $[a,b]$ and $g$ is
continuous, by (\ref{keyeq1}) we have
\begin{align}
\label{eq2.16}\left| {p\left( t \right)} \right| &\le \left( {1 -
\alpha } \right)\left| {\int_a^t {g\left( s \right)du\left( s
\right)} } \right| + \alpha \left| {\int_x^t {g\left( s
\right)du\left( s \right)} } \right|
\nonumber\\
&\le \left( {1 - \alpha } \right)\mathop {\sup }\limits_{s \in
\left[ {a,t} \right]} \left| {g\left( s \right)} \right| \cdot
\bigvee_a^t \left( u \right) + \alpha \mathop {\sup }\limits_{s
\in \left[ {t,x} \right]} \left| {g\left( s \right)} \right| \cdot
\bigvee_t^x \left( u \right)
\nonumber\\
&\le \max \left\{ {\left( {1 - \alpha } \right)\mathop {\sup
}\limits_{s \in \left[ {a,t} \right]} \left| {g\left( s \right)}
\right|,\alpha \mathop {\sup }\limits_{s \in \left[ {t,x} \right]}
\left| {g\left( s \right)} \right|} \right\}\cdot \bigvee_a^x
\left( u \right)
\nonumber\\
&:= M\left( t \right) \cdot \bigvee_a^x \left( u \right).
\end{align}
Similarly, we have
\begin{align}
\label{eq2.17}\left| {q\left( t \right)} \right| &\le \max \left\{
{\left( {1 - \alpha } \right)\mathop {\sup }\limits_{s \in \left[
{t,b} \right]} \left| {g\left( s \right)} \right|,\alpha \mathop
{\sup }\limits_{s \in \left[ {t,x} \right]} \left| {g\left( s
\right)} \right|} \right\}\cdot \bigvee_x^b \left( u \right)
\nonumber\\
&:= N\left( t \right) \cdot \bigvee_x^b \left( u \right).
\end{align}
Thus, by (\ref{eq2.15})--(\ref{eq2.17}), we have
\begin{align*}
\left| {\int_a^b {K_{g,u} \left( {t;x} \right)df\left( t
\right)}}\right| &\le L \left[ {\int_a^x {\left| {p\left( {t}
\right)}\right|dt} + \int_x^b {\left| {q\left( {t}
\right)}\right|dt}}\right]
\\
&\le L \left[ {\left( {\int_a^x {M\left( t \right)dt} } \right)
\cdot \bigvee_a^x \left( u \right) + \left( {\int_x^b {N\left( t
\right)dt} } \right) \cdot \bigvee_x^b \left( u \right)}\right]
\\
&\le L \left[ {\left( {x - a} \right) \cdot\mathop {\sup
}\limits_{t \in \left[ {a,x} \right]} \left\{ {M\left( t \right)}
\right\} \cdot \bigvee_a^x \left( u \right) + \left( {b - x}
\right) \cdot\mathop {\sup }\limits_{t \in \left[ {x,b} \right]}
\left\{ {N\left( t \right)} \right\} \cdot \bigvee_x^b \left( u
\right)}\right]
\\
&\le L \cdot \max \left\{ {\left( {x - a} \right) \cdot\mathop
{\sup }\limits_{t \in \left[ {a,x} \right]} \left\{ {M\left( t
\right)} \right\},\left( {b - x} \right) \cdot\mathop {\sup
}\limits_{t \in \left[ {x,b} \right]} \left\{ {N\left( t \right)}
\right\}} \right\} \cdot \bigvee_a^b \left( u \right),
\end{align*}
which gives the result.
\end{proof}

\begin{remark}
In Theorem \ref{thm2}, if $g\left( t \right) = 1$ for all $t \in
[a,b]$. Then $M\left( t \right) = N\left( t \right) = \left[
{\frac{1}{2} + \left| {\frac{1}{2} - \alpha } \right|} \right]$,
for all $t \in [a,b]$.
\end{remark}

\begin{corollary}
\label{cor2}In Theorem \ref{thm2}, let $g\left( t \right) = 1$ for
all $t \in [a,b]$. Then, we have the inequality
\begin{multline}
\label{eq3.6}\left| {\alpha \left[ {\left( {u\left( x \right) -
u\left( a \right)} \right)f\left( a \right) + \left( {u\left( b
\right) - u\left( x \right)} \right)f\left( b \right)} \right] + \left( {1 - \alpha } \right)\left\{ {\left[ {u\left( {\frac{{a + b}}{2}} \right) - u\left( a \right)} \right]f\left( x \right)} \right.}
\right.
\\
\left. {+ \left. {\left[ {u\left( b \right) - u\left( {\frac{{a + b}}{2}} \right)} \right]f\left( {a + b - x} \right)} \right\}
- \int_a^b {f\left( t
\right)du\left( t \right)}}\right|
\\
\le L\left[ {\frac{1}{2} + \left| {\frac{1}{2} - \alpha } \right|}
\right] \cdot \left[ {\frac{{b - a}}{2} + \left| {x - \frac{{a +
b}}{2}} \right|} \right] \cdot \bigvee_a^b \left( u \right).
\end{multline}
The constant $\left[ {\frac{1}{2} + \left| {\frac{1}{2} - \alpha }
\right|} \right]$ is the best possible.

For instance,
\begin{itemize}
\item If $\alpha = 0$, then we get
\begin{multline}
\left| {\left\{ {\left[ {u\left( {\frac{{a + b}}{2}} \right) - u\left( a \right)} \right]f\left( x \right)
+ \left[ {u\left( b \right) - u\left( {\frac{{a + b}}{2}} \right)} \right]f\left( {a + b - x} \right)} \right\}
- \int_a^b {f\left( t
\right)du\left( t \right)}}\right|
\\
\le L \left[ {\frac{{b - a}}{2} + \left| {x - \frac{{a + b}}{2}}
\right|} \right] \cdot \bigvee_a^b \left( u \right).
\end{multline}

\item If $\alpha = \frac{1}{3}$, then we get
\begin{multline}
\label{eq3.8}\left| {\frac{1}{3} \left[ {\left( {u\left( x \right) -
u\left( a \right)} \right)f\left( a \right) + \left( {u\left( b
\right) - u\left( x \right)} \right)f\left( b \right)} \right] + \frac{2}{3}\left\{ {\left[ {u\left( {\frac{{a + b}}{2}} \right) - u\left( a \right)} \right]f\left( x \right)} \right.}
\right.
\\
\left. {+ \left. {\left[ {u\left( b \right) - u\left( {\frac{{a + b}}{2}} \right)} \right]f\left( {a + b - x} \right)} \right\}
- \int_a^b {f\left( t
\right)du\left( t \right)}}\right|
\\
\le \frac{2}{3}L \left[ {\frac{{b - a}}{2} + \left| {x - \frac{{a
+ b}}{2}} \right|} \right] \cdot \bigvee_a^b \left( u \right).
\end{multline}

\item If $\alpha = \frac{1}{2}$, then we get
\begin{multline}
\left| {\frac{1}{2}\left\{ {\left( {u\left( x \right) -
u\left( a \right)} \right)f\left( a \right) + \left( {u\left( b
\right) - u\left( x \right)} \right)f\left( b \right)+ \left[ {u\left( {\frac{{a + b}}{2}} \right) - u\left( a \right)} \right]f\left( x \right)} \right.}
\right.
\\
\left. {+ \left. {\left[ {u\left( b \right) - u\left( {\frac{{a + b}}{2}} \right)} \right]f\left( {a + b - x} \right)} \right\}
- \int_a^b {f\left( t
\right)du\left( t \right)}}\right|
\\
\le \frac{1}{2}L\left[ {\frac{{b - a}}{2} + \left| {x - \frac{{a +
b}}{2}} \right|} \right] \cdot \bigvee_a^b \left( u \right).
\end{multline}

\item If $\alpha = 1$, then we get
\begin{multline}
\left| {\left[ {u\left( x \right) - u\left( a \right)}
\right]f\left( a \right) + \left[ {u\left( b \right) - u\left( x
\right)} \right]f\left( b \right)- \int_a^b {f\left( t
\right)du\left( t \right)}}\right|
\\
\le L \left[ {\frac{{b - a}}{2} + \left| {x - \frac{{a + b}}{2}}
\right|} \right] \cdot \bigvee_a^b \left( u \right).
\end{multline}
\end{itemize}
\end{corollary}
\begin{proof}
The results follow by Theorem \ref{thm2}. It remains to prove the
sharpness of (\ref{eq3.6}). Suppose $0 \le \alpha \le
\frac{1}{2}$, assume that (\ref{eq3.6}) holds with constant
$C_1>0$, i.e.,
\begin{multline}
\label{eqKKK}\left| {\alpha \left[ {\left( {u\left( x \right) -
u\left( a \right)} \right)f\left( a \right) + \left( {u\left( b
\right) - u\left( x \right)} \right)f\left( b \right)} \right] + \left( {1 - \alpha } \right)\left\{ {\left[ {u\left( {\frac{{a + b}}{2}} \right) - u\left( a \right)} \right]f\left( x \right)} \right.}
\right.
\\
\left. {+ \left. {\left[ {u\left( b \right) - u\left( {\frac{{a + b}}{2}} \right)} \right]f\left( {a + b - x} \right)} \right\}
- \int_a^b {f\left( t
\right)du\left( t \right)}}\right|
\\
\le L C_1 \left[ {\frac{{b - a}}{2} + \left| {x - \frac{{a +
b}}{2}} \right|} \right] \cdot \bigvee_a^b \left( u \right).
\end{multline}
Let $f,u:[a,b] \to \mathbb{R}$ be defined as follows $f\left( t
\right) = t -b$ and
\begin{align*}
u\left( t \right) = \left\{
\begin{array}{l}
 0,\,\,\,\,\,\,\,\,t \in \left[ {a,b} \right)\\
  \\
 1,\,\,\,\,\,\,\,t = b \\
 \end{array} \right.,
\end{align*}
Therefore, $f$ is $L$--Lipschitz with $L=1$ and $\bigvee_a^b\left(
u \right) = 1$ and $\int_a^b {f\left( t \right)du\left( t \right)}
= 0$, setting $x = \frac{a+b}{2}$ it gives by (\ref{eqKKK})
\begin{align*}
\left( {1 - \alpha } \right)\frac{\left( {b-a} \right)}{2} \le
C_1\frac{\left( {b-a} \right)}{2}.
\end{align*}
which proves that $C_1 \ge 1 - \alpha$, and therefore $1- \alpha$
is the best possible for all $0 \le \alpha \le \frac{1}{2}$.

Now, suppose $\frac{1}{2} \le \alpha \le 1$ and assume that
(\ref{eq3.6}) holds with constant $C_2>0$, i.e.,
\begin{multline}
\label{eqKKKK}\left| {\alpha \left[ {\left( {u\left( x \right) -
u\left( a \right)} \right)f\left( a \right) + \left( {u\left( b
\right) - u\left( x \right)} \right)f\left( b \right)} \right] + \left( {1 - \alpha } \right)\left\{ {\left[ {u\left( {\frac{{a + b}}{2}} \right) - u\left( a \right)} \right]f\left( x \right)} \right.}
\right.
\\
\left. {+ \left. {\left[ {u\left( b \right) - u\left( {\frac{{a + b}}{2}} \right)} \right]f\left( {a + b - x} \right)} \right\}
- \int_a^b {f\left( t
\right)du\left( t \right)}}\right|
\\
\le  L C_2 \left[ {\frac{{b - a}}{2} + \left| {x - \frac{{a +
b}}{2}} \right|} \right] \cdot \bigvee_a^b \left( u \right).
\end{multline}
Let $f,u:[a,b] \to \mathbb{R}$ be defined as follows $f\left( t
\right) = t - a$ and
\begin{align*}
u\left( t \right) = \left\{
\begin{array}{l}
 0,\,\,\,\,\,\,\,\,t \in \left[ {a,b} \right]\backslash \left\{ {\frac{a+b}{2}} \right\} \\
  \\
 \frac{1}{2},\,\,\,\,\,\,\,\,t = \frac{a+b}{2} \\
 \end{array} \right..
\end{align*}
Therefore, $f$ is $L$--Lipschitz with $L=1$ and $\bigvee_a^b\left(
u \right) = 1$ and $\int_a^b {f\left( t \right)du\left( t \right)}
= 0$, setting $x = \frac{a+b}{2}$ it gives by (\ref{eqKKKK})
\begin{align*}
\alpha \frac{\left( {b-a} \right)}{2} \le C_2\frac{\left( {b-a}
\right)}{2}.
\end{align*}
which proves that $C_2 \ge \alpha$, and therefore $\alpha$ is the
best possible for all $\frac{1}{2} \le \alpha \le 1$.
Consequently, we can conclude that the constant $\left[
{\frac{1}{2} + \left| {\frac{1}{2} - \alpha } \right|} \right]$ is
the best possible, for all $\alpha\in [0,1]$.
\end{proof}

\begin{corollary}
In (\ref{eq3.8}), choosing $x = \frac{a+b}{2}$, then we have the
following Simpson's type inequality for $\mathcal{RS}$--integrals:
\begin{multline}
\left| {\frac{1}{3}\left\{ {\left[ {u\left( {\frac{a+b}{2}}
\right) - u\left( a \right)} \right]f\left( a \right) + 2 \left[
{u\left( b \right) - u\left( a \right)} \right]f\left(
{\frac{a+b}{2}} \right) } \right.} \right.
\\
\left. {\left. {+ \left[ {u\left( b \right) - u\left(
{\frac{a+b}{2}} \right)} \right]f\left( b \right)} \right\} -
\int_a^b {f\left( t \right)du\left( t \right)}}\right|
\\
\le \frac{1}{3}L \left( {b-a} \right) \cdot \bigvee_a^b \left( u
\right).
\end{multline}
The constant $\frac{1}{3}$ is the best possible.
\end{corollary}

\begin{corollary}
In (\ref{eq3.6}), let $u(t) = t$ for all $t \in [a,b]$, then we
get
\begin{multline}
\left| {\alpha \left( {\left( {x - a} \right)f\left( a \right) +
\left( {b - x} \right)f\left( b \right)} \right) +
\frac{1}{2}\left( {1 - \alpha } \right)\left( {b - a} \right)\left( {f\left( x \right) + f\left( {a + b - x} \right)} \right)
- \int_a^b
{f\left( t \right)dt}}\right|
\\
\le L \left( {b-a} \right) \left[ {\frac{1}{2} + \left|
{\frac{1}{2} - \alpha } \right|} \right] \cdot\left[ {\frac{{b -
a}}{2} + \left| {x - \frac{{a + b}}{2}} \right|} \right].
\end{multline}
For $x = \frac{a+b}{2}$, we have
\begin{multline}
\left| {\left( {b - a} \right) \left[ {\alpha  \frac{f\left( a
\right) + f\left( b \right)}{2} + \left( {1 - \alpha }
\right)f\left( {\frac{a+b}{2}} \right)} \right]- \int_a^b {f\left(
t \right)dt}}\right|
\\
\le L \left[ {\frac{1}{2} + \left| {\frac{1}{2} - \alpha }
\right|} \right] \cdot\frac{\left( {b-a} \right)^2}{2}.
\end{multline}
\end{corollary}

\section{On monotonic nondecreasing integrators}
\begin{theorem}\label{thm3}
Let $g,u$ be as in Theorem \ref{thm2}.  If $f:[a,b] \to
\mathbb{R}$ is monotonic nondecreasing on $[a,b]$, then for any $x
\in [a,b]$ and $\alpha \in [0,1]$, we have
\begin{multline}
\label{eq2.23}\left| {\left( {1 - \alpha } \right)\left[ {f\left( x
\right)\int_a^{{\textstyle{{a + b} \over 2}}} {g\left( s
\right)du\left( s \right)}  + f\left( {a + b - x}
\right)\int_{{\textstyle{{a + b} \over 2}}}^b {g\left( s
\right)du\left( s \right)} } \right]} \right.
\\
\left. {+ \alpha \left[ {f\left( a \right)\int_a^x {g\left( s
\right)du\left( s \right)}  + f\left( b \right)\int_x^b {g\left( s
\right)du\left( s \right)} } \right] - \int_a^b {f\left( t
\right)g\left( t \right)du\left( t \right)}}\right|
\\
\le \mathop {\sup }\limits_{t \in \left[ {a,x} \right]} \left\{
{M\left( t \right)} \right\} \cdot\left[ {f\left( x \right) -
f\left( a \right)} \right] \cdot \bigvee_a^x \left( u \right) +
\mathop {\sup }\limits_{t \in \left[ {x,b} \right]} \left\{
{N\left( t \right)} \right\} \cdot \left[ {f\left( b \right) -
f\left( x \right)} \right] \cdot \bigvee_x^b \left( u \right)
\end{multline}
where, $M\left( t \right)$ and $N\left( t \right)$ are defined in
Theorem \ref{thm2}.
\end{theorem}

\begin{proof}
Using the identity
\begin{align*}
&\alpha \left[ {f\left( a \right)\int_a^x {g\left( s
\right)du\left( s \right)}  + f\left( b \right)\int_x^b {g\left( s
\right)du\left( s \right)} } \right]+ \left( {1 - \alpha }
\right)f\left( x \right)\int_a^b {g\left( s \right)du\left( s
\right)}
\\
&\qquad  - \int_a^b {f\left( t \right)g\left( t \right)du\left( t
\right)}
\\
&=\int_a^b {K_{g,u} \left( {t;x} \right)df\left( t \right)}.
\end{align*}
It is well-known that for a monotonic non-decreasing function
$\nu:[a,b] \to \mathbb{R}$ and continuous function $p:[a,b] \to
\mathbb{R}$, one has the inequality
\begin{align}
\label{keyeq3}\left| {\int_a^b {p\left( t \right)d\nu\left( t
\right)} } \right| \le \int_a^b {\left| {p\left( t \right)}
\right|d\nu \left( t \right)}.
\end{align}
As $f$ is monotonic non-decreasing on $[a,b]$, by (\ref{keyeq3})
we have
\begin{align}
\label{eq2.25}\left| {\int_a^b {K_{g,u} \left( {t;x}
\right)df\left( t \right)}}\right| &\le \int_a^b {\left| {K_{g,u}
\left( {t;x} \right)}\right|df\left( t \right)}
\nonumber\\
&=  \int_a^x {\left| {p\left( {t} \right)}\right|df\left( t
\right)} + \int_x^b {\left| {q\left( {t} \right)}\right|df\left( t
\right)}
\end{align}
Now, as $u$ is of bounded variation on $[a,b]$ and $g$ is
continuous, by (\ref{eq2.16})--(\ref{eq2.17}) we have
\begin{align}
\label{eq2.26}\left| {p\left( t \right)} \right| \le  M\left( t
\right) \cdot \bigvee_a^x \left( u
\right),\,\,\,\,\,\,\,\,\,\,\,\,\,\,\,\left| {q\left( t \right)}
\right| \le  N\left( t \right) \cdot \bigvee_x^b \left( u \right)
\end{align}
Thus, by (\ref{eq2.25}) and (\ref{eq2.26}), we have
\begin{align*}
&\left| {\int_a^b {K_{g,u} \left( {t;x} \right)df\left( t
\right)}}\right|
\\
&\le \int_a^x {\left| {p\left( {t} \right)}\right|df\left( t
\right)} + \int_x^b {\left| {q\left( {t} \right)}\right|df\left( t
\right)}
\\
&\le \left( {\int_a^x {M\left( t \right)df\left( t \right)} }
\right) \cdot \bigvee_a^x \left( u \right) + \left( {\int_x^b
{N\left( t \right)df\left( t \right)} } \right) \cdot \bigvee_x^b
\left( u \right)
\\
&\le \mathop {\sup }\limits_{t \in \left[ {a,x} \right]} \left\{
{M\left( t \right)} \right\} \cdot\left[ {f\left( x \right) -
f\left( a \right)} \right] \cdot \bigvee_a^x \left( u \right) +
\mathop {\sup }\limits_{t \in \left[ {x,b} \right]} \left\{
{N\left( t \right)} \right\} \cdot \left[ {f\left( b \right) -
f\left( x \right)} \right] \cdot \bigvee_x^b \left( u \right)
\end{align*}
which gives the result.
\end{proof}

\begin{corollary}
\label{cor3}In Theorem \ref{thm3}, let $g\left( t \right) = 1$ for
all $t \in [a,b]$. Then, we have the inequality
\begin{multline}
\label{eq4.5}\left| {\alpha \left[ {\left( {u\left( x \right) -
u\left( a \right)} \right)f\left( a \right) + \left( {u\left( b
\right) - u\left( x \right)} \right)f\left( b \right)} \right] + \left( {1 - \alpha } \right)\left\{ {\left[ {u\left( {\frac{{a + b}}{2}} \right) - u\left( a \right)} \right]f\left( x \right)} \right.}
\right.
\\
\left. {+ \left. {\left[ {u\left( b \right) - u\left( {\frac{{a + b}}{2}} \right)} \right]f\left( {a + b - x} \right)} \right\}
- \int_a^b {f\left( t
\right)du\left( t \right)}}\right|
\\
\le \left[ {\frac{1}{2} + \left| {\frac{1}{2} - \alpha } \right|}
\right] \cdot \left\{ {\left[ {f\left( x \right) - f\left( a
\right)} \right] \cdot \bigvee_a^x \left( u \right) + \left[
{f\left( b \right) - f\left( x \right)} \right] \cdot \bigvee_x^b
\left( u \right)}\right\}
\\
\le \left[ {\frac{1}{2} + \left| {\frac{1}{2} - \alpha } \right|}
\right] \cdot \left[ {\frac{{f\left( b \right) - f\left( a
\right)}}{2} + \left| {f\left( x \right) - \frac{{f\left( a
\right) + f\left( b \right)}}{2}} \right|} \right] \cdot
\bigvee_a^b \left( u \right).
\end{multline}
For the last inequality, the constant $\left[ {\frac{1}{2} +
\left| {\frac{1}{2} - \alpha } \right|} \right]$ is the best
possible.

For instance,
\begin{itemize}
\item If $\alpha = 0$, then we get
\begin{multline}
\left| {\left\{ {\left[ {u\left( {\frac{{a + b}}{2}} \right) - u\left( a \right)} \right]f\left( x \right)
+ \left[ {u\left( b \right) - u\left( {\frac{{a + b}}{2}} \right)} \right]f\left( {a + b - x} \right)} \right\}
- \int_a^b {f\left( t
\right)du\left( t \right)}}\right|
\\
\le \left[ {f\left( x \right) - f\left( a \right)} \right] \cdot
\bigvee_a^x \left( u \right) + \left[ {f\left( b \right) - f\left(
x \right)} \right] \cdot \bigvee_x^b \left( u \right)
\\
\le \left[ {\frac{{f\left( b \right) - f\left( a \right)}}{2} +
\left| {f\left( x \right) - \frac{{f\left( a \right) + f\left( b
\right)}}{2}} \right|} \right] \cdot \bigvee_a^b \left( u \right).
\end{multline}

\item If $\alpha = \frac{1}{3}$, then we get
\begin{multline}
\label{eq4.7}\left| {\frac{1}{3} \left[ {\left( {u\left( x \right) -
u\left( a \right)} \right)f\left( a \right) + \left( {u\left( b
\right) - u\left( x \right)} \right)f\left( b \right)} \right] + \frac{2}{3}\left\{ {\left[ {u\left( {\frac{{a + b}}{2}} \right) - u\left( a \right)} \right]f\left( x \right)} \right.}
\right.
\\
\left. {+ \left. {\left[ {u\left( b \right) - u\left( {\frac{{a + b}}{2}} \right)} \right]f\left( {a + b - x} \right)} \right\}
- \int_a^b {f\left( t
\right)du\left( t \right)}}\right|
\\
\le \frac{2}{3}\left\{ { \left[ {f\left( x \right) - f\left( a
\right)} \right] \cdot \bigvee_a^x \left( u \right) + \left[
{f\left( b \right) - f\left( x \right)} \right] \cdot \bigvee_x^b
\left( u \right)}\right\}
\\
\le \frac{2}{3} \left[ {\frac{{f\left( b \right) - f\left( a
\right)}}{2} + \left| {f\left( x \right) - \frac{{f\left( a
\right) + f\left( b \right)}}{2}} \right|} \right] \cdot
\bigvee_a^b \left( u \right).
\end{multline}

\item If $\alpha = \frac{1}{2}$, then we get
\begin{multline}
\left| {\frac{1}{2}\left\{ {\left( {u\left( x \right) -
u\left( a \right)} \right)f\left( a \right) + \left( {u\left( b
\right) - u\left( x \right)} \right)f\left( b \right)+ \left[ {u\left( {\frac{{a + b}}{2}} \right) - u\left( a \right)} \right]f\left( x \right)} \right.}
\right.
\\
\left. {+ \left. {\left[ {u\left( b \right) - u\left( {\frac{{a + b}}{2}} \right)} \right]f\left( {a + b - x} \right)} \right\}
- \int_a^b {f\left( t
\right)du\left( t \right)}}\right|
\\
\le \frac{1}{2}\left\{ { \left[ {f\left( x \right) - f\left( a
\right)} \right] \cdot \bigvee_a^x \left( u \right) + \left[
{f\left( b \right) - f\left( x \right)} \right] \cdot \bigvee_x^b
\left( u \right)}\right\}
\\
\le \frac{1}{2} \left[ {\frac{{f\left( b \right) - f\left( a
\right)}}{2} + \left| {f\left( x \right) - \frac{{f\left( a
\right) + f\left( b \right)}}{2}} \right|} \right] \cdot
\bigvee_a^b \left( u \right).
\end{multline}

\item If $\alpha = 1$, then we get
\begin{multline}
\left| {\left[ {u\left( x \right) - u\left( a \right)}
\right]f\left( a \right) + \left[ {u\left( b \right) - u\left( x
\right)} \right]f\left( b \right)- \int_a^b {f\left( t
\right)du\left( t \right)}}\right|
\\
\le \left[ {f\left( x \right) - f\left( a \right)} \right] \cdot
\bigvee_a^x \left( u \right) + \left[ {f\left( b \right) - f\left(
x \right)} \right] \cdot \bigvee_x^b \left( u \right)
\\
\le \left[ {\frac{{f\left( b \right) - f\left( a \right)}}{2} +
\left| {f\left( x \right) - \frac{{f\left( a \right) + f\left( b
\right)}}{2}} \right|} \right] \cdot \bigvee_a^b \left( u \right).
\end{multline}
\end{itemize}
\end{corollary}

\begin{proof}
The results follow by Theorem \ref{thm3}. It remains to prove the
sharpness of (\ref{eq4.5}). Suppose $0 \le \alpha \le
\frac{1}{2}$, assume that (\ref{eq4.5}) holds with constant
$C_1>0$, i.e.,
\begin{multline}
\label{eqL}\left| {\alpha \left[ {\left( {u\left( x \right) -
u\left( a \right)} \right)f\left( a \right) + \left( {u\left( b
\right) - u\left( x \right)} \right)f\left( b \right)} \right] + \left( {1 - \alpha } \right)\left\{ {\left[ {u\left( {\frac{{a + b}}{2}} \right) - u\left( a \right)} \right]f\left( x \right)} \right.}
\right.
\\
\left. {+ \left. {\left[ {u\left( b \right) - u\left( {\frac{{a + b}}{2}} \right)} \right]f\left( {a + b - x} \right)} \right\}
- \int_a^b {f\left( t
\right)du\left( t \right)}}\right|
\\
\le C_1 \left[ {\frac{{f\left( b \right) - f\left( a \right)}}{2}
+ \left| {f\left( x \right) - \frac{{f\left( a \right) + f\left( b
\right)}}{2}} \right|} \right] \cdot \bigvee_a^b \left( u \right).
\end{multline}
Let $f,u:[a,b] \to \mathbb{R}$ be defined as follows
\begin{align*}
f\left( t \right) = \left\{
\begin{array}{l}
 -1,\,\,\,\,\,\,\,\,t =a\\
 \\
 0,\,\,\,\,\,\,\,\,\,\,\,\,\,t = (a,b] \\
 \end{array} \right.,
\end{align*}
and
\begin{align*}
u\left( t \right) = \left\{
\begin{array}{l}
 0,\,\,\,\,\,\,\,t \in \left[ {a,b} \right)\\
  \\
 1,\,\,\,\,\,\,\,t = b \\
 \end{array} \right.,
\end{align*}
Therefore, $f$ is monotonic nondecreasing on $[a,b]$ and
$\bigvee_a^b\left( u \right) = 1$ and $\int_a^b {f\left( t
\right)du\left( t \right)} = 0$, setting $x = a$ it gives by
(\ref{eqL}) that $1 - \alpha \le C_1 $, and which proves that $1-
\alpha$ is the best possible for all $0 \le \alpha \le
\frac{1}{2}$.

Now, suppose $\frac{1}{2} \le \alpha \le 1$ and assume that
(\ref{eq4.5}) holds with constant $C_2>0$, i.e.,
\begin{multline}
\label{eqLL}\left| {\alpha \left[ {\left( {u\left( x \right) -
u\left( a \right)} \right)f\left( a \right) + \left( {u\left( b
\right) - u\left( x \right)} \right)f\left( b \right)} \right] }
\right.
\\
\left. {+ \left( {1 - \alpha } \right)\left[ {u\left( b \right) -
u\left( a \right)} \right]f\left( x \right)- \int_a^b {f\left( t
\right)du\left( t \right)}}\right|
\\
\le  C_2 \left[ {\frac{{f\left( b \right) - f\left( a \right)}}{2}
+ \left| {f\left( x \right) - \frac{{f\left( a \right) + f\left( b
\right)}}{2}} \right|} \right] \cdot \bigvee_a^b \left( u \right).
\end{multline}
Let $f,u:[a,b] \to \mathbb{R}$ be defined as  $f\left( t \right)$
as above, and $u\left( t \right) = t$, which follows that
$\bigvee_a^b\left( u \right) = b-a$, and $\int_a^b {f\left( t
\right)du\left( t \right)}  = 0$, setting $x = b$ it gives by
(\ref{eqLL}) $\alpha \le C_2$, and therefore $\alpha$ is the best
possible for all $\frac{1}{2} \le \alpha \le 1$. Consequently, we
can conclude that the constant $\left[ {\frac{1}{2} + \left|
{\frac{1}{2} - \alpha } \right|} \right]$ is the best possible,
for all $\alpha\in [0,1]$.
\end{proof}

\begin{corollary}
In (\ref{eq4.7}), choosing $x = \frac{a+b}{2}$, then we have the
following Simpson's type inequality for $\mathcal{RS}$--integrals:
\begin{multline}
\left| {\frac{1}{3}\left\{ {\left[ {u\left( {\frac{a+b}{2}}
\right) - u\left( a \right)} \right]f\left( a \right) + 2 \left[
{u\left( b \right) - u\left( a \right)} \right]f\left(
{\frac{a+b}{2}} \right) } \right.} \right.
\\
\left. {\left. {+ \left[ {u\left( b \right) - u\left(
{\frac{a+b}{2}} \right)} \right]f\left( b \right)} \right\} -
\int_a^b {f\left( t \right)du\left( t \right)}}\right|
\\
\le \frac{2}{3}\left\{ { \left[ {f\left( {\frac{a+b}{2}} \right) -
f\left( a \right)} \right] \cdot \bigvee_a^{{\textstyle{{a + b}
\over 2}}} \left( u \right) + \left[ {f\left( b \right) - f\left(
{\frac{a+b}{2}} \right)} \right] \cdot \bigvee_{{\textstyle{{a +
b} \over 2}}}^b \left( u \right)}\right\}
\\
\le \frac{2}{3} \left[ {\frac{{f\left( b \right) - f\left( a
\right)}}{2} + \left| {f\left( {\frac{a+b}{2}} \right) -
\frac{{f\left( a \right) + f\left( b \right)}}{2}} \right|}
\right] \cdot \bigvee_a^b \left( u \right).
\end{multline}
For the last inequality, the constant $\frac{2}{3}$ is the best
possible.
\end{corollary}

\begin{corollary}
In (\ref{eq4.5}), let $u(t) = t$ for all $t \in [a,b]$, then we
get
\begin{multline}
\left| {\alpha \left( {\left( {x - a} \right)f\left( a \right) +
\left( {b - x} \right)f\left( b \right)} \right) +
\frac{1}{2}\left( {1 - \alpha } \right)\left( {b - a} \right)\left( {f\left( x \right) + f\left( {a + b - x} \right)} \right)
- \int_a^b
{f\left( t \right)dt}}\right|
\\
\le \left[ {\frac{1}{2} + \left| {\frac{1}{2} - \alpha } \right|}
\right] \cdot \left\{ {\left( {x-a} \right) \cdot \left[ {f\left(
x \right) - f\left( a \right)} \right] + \left( {b-x} \right)
\cdot \left[ {f\left( b \right) - f\left( x \right)} \right]
}\right\}
\\
\le \left[ {\frac{1}{2} + \left| {\frac{1}{2} - \alpha } \right|}
\right] \cdot \left[ {\frac{{f\left( b \right) - f\left( a
\right)}}{2} + \left| {f\left( x \right) - \frac{{f\left( a
\right) + f\left( b \right)}}{2}} \right|} \right] \cdot \left(
{b-a} \right).
\end{multline}
For $x = \frac{a+b}{2}$, we have
\begin{align}
&\left| {\left( {b - a} \right) \left[ {\alpha  \frac{f\left( a
\right) + f\left( b \right)}{2} + \left( {1 - \alpha }
\right)f\left( {\frac{a+b}{2}} \right)} \right]- \int_a^b {f\left(
t \right)dt}}\right|
\\
&\le \frac{1}{2}\left( {b-a} \right)\left[ {\frac{1}{2} + \left|
{\frac{1}{2} - \alpha } \right|} \right] \cdot \left[ {f\left( b
\right) - f\left( a \right)} \right]
\nonumber\\
&\le \left[ {\frac{1}{2} + \left| {\frac{1}{2} - \alpha } \right|}
\right] \cdot \left[ {\frac{{f\left( b \right) - f\left( a
\right)}}{2} + \left| {f\left( {\frac{a+b}{2}} \right) -
\frac{{f\left( a \right) + f\left( b \right)}}{2}} \right|}
\right] \cdot \left( {b-a} \right).\nonumber
\end{align}
\end{corollary}

\begin{remark}
We give an attention to the interested reader, is that, in
Theorems \ref{thm1}--\ref{thm3}, one may observe various new
inequalities by replacing the assumptions on $u$, e.g. to be of
bounded variation, $L_u$--Lipschitz or monotonic nondecreasing on
$[a,b]$, which therefore gives in some cases the `dual' of the
above obtained inequalities.

It remains to mention that, in  Theorem \ref{thm2}, and according
to the assumptions on $u$ one may observe several estimations for
the functions $p\left( {t} \right)$ and $q\left( {t} \right)$
which therefore gives different functions $M\left( {t} \right)$
and $N\left( {t} \right)$.
\end{remark}

\begin{remark}
In Theorems \ref{thm1}--\ref{thm3}, a different result(s) in terms
of $L_p$ norms may be stated by applying the well--known
H\"{o}lder integral inequality, by noting that
\begin{align*}
\left| {\int_c^d {g\left( s \right)du\left( s \right)} } \right|
\le \sqrt[q]{{u\left( d \right) - u\left( c \right)}} \times
\sqrt[p]{{\int_c^d {\left| {g\left( s \right)} \right|^p du\left(
s \right)} }}.
\end{align*}
where, $p>1$, $\frac{1}{p} + \frac{1}{q} = 1$.
\end{remark}

\begin{remark}
One can point out some results for the Riemann integral of a
product, in terms of $L_1$, $L_p$ and $L_{\infty}$ norms by using
a similar argument considered in \cite{Dragomir2} (see also
\cite{alomari1}--\cite{alomari2}).
\end{remark}

\section{Applications to Ostrowski-generalized trapezoid quadrature formula for $\mathcal{RS}$-integrals}

Let $I_n : a = x_0 < x_1 < \cdots < x_n = b$ be a division of the
interval $[a,b]$. Define the general Riemann--Stieltjes sum

\begin{multline}
\label{eq4.1}S\left( {f,u,I_n ,\xi } \right)= \sum\limits_{i =
0}^{n - 1}
\alpha \left[ {\left( {u\left( {\xi _i } \right) - u\left( {x_i } \right)} \right)f\left( {x_i } \right) + \left( {u\left( {x_{i + 1} } \right) - u\left( {\xi _i } \right)} \right)f\left( {x_{i + 1} } \right)} \right]
\\
+ \left( {1 - \alpha } \right)\left\{ {\left[ {u\left( {\frac{{x_i  + x_{i + 1} }}{2}} \right) - u\left( {x_i } \right)} \right]f\left( {\xi _i } \right) }\right.
\\
+ \left.{\left[ {u\left( {x_{i + 1} } \right) - u\left( {\frac{{x_i  + x_{i + 1} }}{2}} \right)} \right]f\left( {x_i  + x_{i + 1}  - \xi _i } \right)} \right\}
\end{multline}
In the following, we establish an upper bound for the error
approximation of the Riemann-Stieltjes integral $\int_a^b {f\left(
t \right)du\left( t \right)}$ by its Riemann-Stieltjes sum
$S\left( {f,u,I_n ,\xi } \right)$. As a sample we apply the
inequality (\ref{eq2.8}).

\begin{theorem}
\label{thm5} Under the assumptions of Corollary \ref{cor1}, we
have
\begin{align*}
\int_a^b {f\left( t \right)du\left( t \right)}  = S\left( {f,u,I_n
,\xi } \right) + R\left( {f,u,I_n ,\xi } \right)
\end{align*}
where, $S\left( {f,u,I_n ,\xi } \right)$ is given in (\ref{eq4.1})
and the remainder $R\left( {f,u,I_n ,\xi } \right)$ satisfies the
bound
\begin{align}
\left| {R\left( {f,u,I_n ,\xi } \right)} \right| &\le \left[
{\frac{1}{2} + \left| {\frac{1}{2} - \alpha } \right|} \right]
\left[ {u\left( b \right) - u\left( a \right)} \right]\cdot
\bigvee_{a}^{b} \left( f \right).
\end{align}
\end{theorem}

\begin{proof}
Applying Corollary \ref{cor1} on the intervals $[x_i, x_{i+1}]$,
we may state that
\begin{multline*}
\left| {\alpha \left[ {\left( {u\left( {\xi _i } \right) - u\left( {x_i } \right)} \right)f\left( {x_i } \right) + \left( {u\left( {x_{i + 1} } \right) - u\left( {\xi _i } \right)} \right)f\left( {x_{i + 1} } \right)} \right]
+ \left( {1 - \alpha } \right)\left\{ {\left[ {u\left( {\frac{{x_i  + x_{i + 1} }}{2}} \right) - u\left( {x_i } \right)} \right]f\left( {\xi _i } \right) }\right.} \right.
\\
\left. {\left.{\left[ {u\left( {x_{i + 1} } \right) - u\left( {\frac{{x_i  + x_{i + 1} }}{2}} \right)} \right]f\left( {x_i  + x_{i + 1}  - \xi _i } \right)} \right\}
 - \int_{x_i }^{x_{i + 1} } {f\left( t \right)du\left( t \right)}
} \right|
\\
\le \left[ {\frac{1}{2} + \left| {\frac{1}{2} - \alpha } \right|}
\right]\left[ {\frac{{u\left( {x_{i + 1} } \right) - u\left( {x_i
} \right)}}{2} + \left| {u\left( {\xi _i } \right) -
\frac{{u\left( {x_i } \right) + u\left( {x_{i + 1} } \right)}}{2}}
\right|} \right] \cdot \bigvee_{x_i}^{x_{i+1}} \left( f \right),
\end{multline*}
for all $i \in \{{0,1,2, \cdots, n-1}\}$.

Summing the above inequality over $i$ from $0$ to $n - 1$ and
using the generalized triangle inequality, we deduce
\begin{align*}
&\left| {R\left( {f,u,I_n ,\xi } \right)} \right|
\\
&\le \left[ {\frac{1}{2} + \left| {\frac{1}{2} - \alpha } \right|}
\right]\sum\limits_{i = 0}^{n - 1}  \left[ {\frac{{u\left( {x_{i +
1} } \right) - u\left( {x_i } \right)}}{2} + \left| {u\left( {\xi
_i } \right) - \frac{{u\left( {x_i } \right) + u\left( {x_{i + 1}
} \right)}}{2}} \right|} \right] \cdot \bigvee_{x_i}^{x_{i+1}}
\left( f \right)
\\
&\le \left[ {\frac{1}{2} + \left| {\frac{1}{2} - \alpha } \right|}
\right] \left[ {\sum\limits_{i = 0}^{n - 1}\frac{{u\left( {x_{i +
1} } \right) - u\left( {x_i } \right)}}{2} + \sum\limits_{i =
0}^{n - 1}\left| {u\left( {\xi _i } \right) - \frac{{u\left( {x_i
} \right) + u\left( {x_{i + 1} } \right)}}{2}} \right|}
\right]\cdot \sum\limits_{i = 0}^{n - 1} {\bigvee_{x_i}^{x_{i+1}}
\left( f \right)}
\\
&\le \left[ {\frac{1}{2} + \left| {\frac{1}{2} - \alpha } \right|}
\right] \left[ {\frac{{u\left( b \right) - u\left( a \right)}}{2}
+ \mathop {\sup }\limits_{i = 0,1, \ldots ,n - 1}\left| {u\left(
{\xi _i } \right) - \frac{{u\left( {x_i } \right) + u\left( {x_{i
+ 1} } \right)}}{2}} \right|} \right]\cdot \bigvee_{a}^{b} \left(
f \right)
\\
&\le \left[ {\frac{1}{2} + \left| {\frac{1}{2} - \alpha } \right|}
\right] \left[ {u\left( b \right) - u\left( a \right)}
\right]\cdot \bigvee_{a}^{b} \left( f \right).
\end{align*}
Since,
\begin{align*}
\mathop {\sup }\limits_{i = 0,1, \ldots ,n - 1}\left| {u\left(
{\xi _i } \right) - \frac{{u\left( {x_i } \right) + u\left( {x_{i
+ 1} } \right)}}{2}} \right| \le \mathop {\sup }\limits_{i = 0,1,
\ldots ,n - 1}  \frac{{u\left( {x_{i + 1} } \right) - u\left( {x_i
} \right)}}{2}= \frac{{u\left( b \right) - u\left( a \right)}}{2}
\end{align*}
and
\begin{align*}
\sum\limits_{i = 0}^{n - 1} {\bigvee_{x_i}^{x_{i+1}} \left( f
\right)} = \bigvee_{a}^{b} \left( f \right).
\end{align*}
which completes the proof.
\end{proof}

\begin{remark}
One may use the remaining inequalities in Section 2, to obtain
other bounds for $R\left( {f,u,I_n ,\xi } \right)$. We shall omit
the details.
\end{remark}

\centerline{}\centerline{}\centerline{}

\textbf{Acknowledgment:} I would like to express my deepest thanks for Dr. Eder Kikianty for appreciated concerns and efforts to provide important comments about the results of this paper.

\centerline{}

\centerline{}


\begin{thebibliography}{9}
\setlength{\itemsep}{5pt}

\bibitem{alomari1}
M.W. Alomari, A companion of Ostrowski's inequality for the
Riemann-Stieltjes integral $\int_a^b {f\left( t \right)du\left( t
\right)}$, where $f$ is of $r$-$H$-H\"{o}lder type and $u$ is of
bounded variation and applications, submitted. Avalibale at:
\url{http://ajmaa.org/RGMIA/papers/v14/v14a59.pdf}.

\bibitem{alomari2}
M.W. Alomari, A companion of Ostrowski's inequality for the
Riemann-Stieltjes integral $\int_a^b {f\left( t \right)du\left( t
\right)}$, where $f$ is of bounded variation and $u$ is of
$r$-$H$-H\"{o}lder type and applications, submitted. Avalibale at:
\url{http://ajmaa.org/RGMIA/papers/v14/v14a65.pdf}.


\bibitem{alomari3}
M.W. Alomari, A three point quadrature rule for the
Riemann--Stieltjes integral with applications, in preparation.

\bibitem{Anastassiou}
G.A. Anastassiou, Univariate Ostrowski inequalities,
\textit{Monatsh. Math.}, 135 (3) (2002) 175--189. Revisited.


\bibitem{Barnett}
N.S. Barnett, S.S. Dragomir and I. Gomma,  A companion for the
Ostrowski and the generalised trapezoid inequalities, \textit{
Math. and Comp. Mode.}, 50 (2009), 179--187.

\bibitem{Barnett1}
N.S. Barnett, W.-S. Cheung, S.S. Dragomir, A. Sofo,  Ostrowski and
trapezoid type inequalities for the Stieltjes integral with
Lipschitzian integrands or integrators, \textit{Comp. Math. Appl.
}, 57 (2009), 195--201.

\bibitem{Cerone}
P. Cerone, W.S. Cheung, S.S. Dragomir, On Ostrowski type
inequalities for Stieltjes integrals with absolutely continuous
integrands and integrators of bounded variation, \textit{Comp.
Math. Appl.}, 54 (2007), 183--191.

\bibitem{CeroneDragomir}
P. Cerone, S.S. Dragomir, New bounds for the three-point rule
involving the Riemann-Stieltjes integrals, in: C. Gulati, et al.
(Eds.), Advances in Statistics Combinatorics and Related Areas,
World Science Publishing, 2002, pp. 53--62.

\bibitem{CeroneDragomir1}
P. Cerone, S.S. Dragomir, Approximating the Riemann--Stieltjes
integral via some moments of the integrand, \textit{Mathematical
and Computer Modelling}, 49 (2009), 242--248.

\bibitem{Cerone}
P. Cerone, S.S. Dragomir, C.E.M. Pearce, A generalized trapezoid
inequality for functions of bounded variation, \textit{Turk. J.
Math.}, 24 (2000), 147--163.

\bibitem{CheungDragomir}
W.-S. Cheung, S.S. Dragomir, Two Ostrowski type inequalities for
the Stieltjes integral of monotonic functions, \textit{Bull.
Austral. Math. Soc.}, 75 (2007), 299--311.

\bibitem{Dragomir1}
S.S. Dragomit, Ostrowski integral inequality for mappings of
bounded variation, \textit{Bull. Austral. Math. Soc.}, 60 (1999)
495--508.

\bibitem{Dragomir2}
S.S. Dragomir, On the Ostrowski inequality for Riemann--Stieltjes
integral $\int_a^b{f(t) du(t)}$ where $f$ is of H\"{o}lder type
and $u$ is of bounded variation and applications, \textit{J.
KSIAM}, 5 (2001), 35--45.

\bibitem{Dragomir3}
S.S. Dragomir, On the Ostrowski's inequality for Riemann--Stieltes
integral and applications, \textit{Korean J. Comput. \& Appl.
Math.}, 7 (2000),  611--627.

\bibitem{Dragomir5}
S.S. Dragomir, C. Bu\c{s}e, M.V. Boldea, L. Braescu, A
generalisation of the trapezoid rule for the Riemann--Stieltjes
integral and applications, \textit{Nonlinear Anal. Forum}, 6 (2)
(2001) 33--351.

\bibitem{Dragomir6}
S.S. Dragomir, Some inequalities of midpoint and trapezoid type
for the Riemann-Stieltjes integral, \textit{Nonlinear Anal.}, 47
(4) (2001), 2333--2340.

\bibitem{Drag}
S.S. Dragomir, Refinements of the generalised trapezoid and
Ostrowski inequalities for functions of bounded variation,
\textit{Arch. Math.},  91 (2008), 450--460

\bibitem{Dragomir7}
S.S. Dragomir, Approximating the Riemann--Stieltjes integral in
terms of generalised trapezoidal rules, \textit{Nonlinear Anal.}
TMA 71 (2009), e62--e72.

\bibitem{Dragomir8}
S.S. Dragomir, Approximating the Riemann--Stieltjes integral by a
trapezoidal quadrature rule with applications,
\textit{Mathematical and Computer Modelling} 54 (2011), 243--260.

\bibitem{Liu}
Z. Liu, Another generalization of weighted Ostrowski type
inequality for mappings of bounded variation, \textit{Appl. Math.
Lett.} 25 (2012) 393--397.

\bibitem{WLiu}
W.-J. Liu, Some weighted integral inequalities with a parameter
and applications, \textit{Acta Appl Math.}, 109 (2010), 389--400.

\bibitem{Tseng1}
K.L. Tseng, S.R. Hwang, S.S. Dragomir, Generalizations of weighted
Ostrowski type inequalities for mappings of bounded variation and
their applications, \textit{Computers and Mathematics with
Applications},  55 (2008) 1785--1793.

\bibitem{Tseng2} K.L. Tseng, Improvements of some inequalites of Ostrowski type
and their applications, \textit{Taiwanese J. Math.}, 12(9)(2008),
2427--2441.

\bibitem{Tseng3}K.L. Tseng, S.R. Hwang, G.S. Yang and Y.M. Chou,
Improvements of the Ostrowski integral inequality for mappings of
bounded variation I, \textit{Appl. Math. Comp.} 217 (2010)
2348--2355.
\end{thebibliography}
\end{document}